\documentclass[12pt]{amsart}
\usepackage{amssymb,amsmath,latexsym,enumerate,graphicx,mathptmx,microtype}
\usepackage{amsfonts}
\usepackage{epstopdf}
\usepackage{verbatim}
\usepackage{amsthm}
\usepackage{delarray}
\usepackage{hyperref}
\usepackage[latin1]{inputenc}
\usepackage[dvips]{color}
\theoremstyle{plain}

\numberwithin{equation}{section}
\newtheorem{theorem}{Theorem}[section]

\newtheorem{lemma}{Lemma}[section]
\newtheorem{corollary}[theorem]{Corollary}

\def\N{\mathbb N}
\def\Z{\mathbb Z}

\def\C{\mathbb C}

\begin{document}
\title[Appell polynomials type]{ New Characterization of Appell  polynomials}
\author{Abdelmejid Bayad and Takao Komatsu$^*$}
\address{Abdelmejid Bayad\\
D\'epartement de math\'ematiques \\
Universit\'e d'Evry Val d'Essonne, 23 Bd. De France\\
 91037 Evry Cedex, France\\}
\email{abayad@maths.univ-evry.fr}
\address{Takao Komatsu\\   
School of Mathematics and Statistics\\
Wuhan University, Wuhan, 430072, China\\
}
\email{komatsu@whu.edu.cn}

\thanks{\footnotesize{\it $^*$Corresponding Author}. komatsu@whu.edu.cn}

\begin{abstract}
We prove  characterizations of Appell polynomials  by means of symmetric property.  For these polynomials, we establish  a simple linear expression in terms of Bernoulli and Euler polynomials.  As applications, we give  interesting examples. In addition, from our study, we obtain Fourier expansions of Appell polynomials. This result recovers Fourier expansions known for Bernoulli and Euler polynomials and  obtains the Fourier expansions for higher order Bernoulli-Euler's one.
\end{abstract}

\maketitle

\section{Introduction}
We give a definition as broad as possible for  the so-called sequences of Appell polynomials. This definition is based on the Appell's original work  \cite{Appell}  published in 1882.\\
Let us fix $g,\varphi$ two functions. The function  $ g:\C\to \C$  is  holomorphic at $0$ , $g(0)=0,  \ g'(0)\neq 0,$ 
and $\varphi:\N \to \C\backslash\{0\}$ an arbitrary function.\\
Let $\left(P_n(x)\right)_n$  be a sequence of polynomials.  We call   $\left(P_n(x)\right)_n$ a sequence of  Appell polynomials of type $(g,\varphi),$  if and only if there exists  $f:\C\to \C$   holomorphic at $0$  such that $f(0)\neq 0$ and 
\begin{eqnarray}\label{Bernoulli-type}
\sum\limits_{n\geq 0}P_{n}(x)\frac{t^{n}}{\varphi(0)\cdots\varphi(n)}=f(t)e^{x g(t)}.
\end{eqnarray}
The ordinary Appell sequence of polynomials corresponds to special type
\[ g(t)=t, \varphi(0)=1,\  \textrm{ and }  \varphi(n)=n,\ n\geq 1.\]
In addition,  for the above type we have Bernoulli's polynomials $B_n(x)$ corresponds to $f_B(t)=\frac{t}{e^t-1}$, and Euler's polynomials corresponds to $f_E(t)=\frac{2}{e^t+1}.$\\

However for this paper,  without loss of generality, we can assume that $\varphi(n)=n$ for $n\geq 1$ and $\varphi(0)=1$. To see this we put  
\[
 Q_n(x)= P_n(x)\frac{n!}{\varphi(0)\cdots\varphi(n)} 
 \]
 and rewrite the relation \eqref{Bernoulli-type}   
\begin{eqnarray}\label{Bernoulli-type2}
f(t)e^{x g(t)}=\sum\limits_{n\geq 0}P_{n}(x)\frac{t^{n}}{\varphi(0)\cdots\varphi(n)}=\sum\limits_{n\geq 0}Q_{n}(x)\frac{t^{n}}{n!}.
\end{eqnarray}
In this paper,  for fixed analytic function $g$ we study the classes  of  sequence of Appell polynomials $\left(P_n(x)\right)_n$ in \eqref{Bernoulli-type} when they satisfy the symmetric relation
\begin{eqnarray}\label{symmetry}
P_{n}(a-x)=(-1)^nP_{n}(x)
\end{eqnarray}
for some  real parameter $a$.
\subsection{Known  Characterizations: An overview}\

We review the known results concerning  special cases of ordinary Appell sequences of polynomials. Namely, Bernoulli and Euler polynomials.
According to Bernoulli, Euler, Appel, Hurwitz, Raabe, Lucas. There are several approaches to  study  Bernoulli and Euler
polynomials.   Here, we list them :
\begin{enumerate}
\item Generating functions theory (Euler \cite{Euler})
\begin{eqnarray}\label{Bernoulli}
\sum\limits_{n\geq 0}B_{n}(x)\frac{t^{n}}{n!}=\frac{t.e^{xt}}{e^{t}-1}%
~~,\left| t\right| <2\pi  \ .
\end{eqnarray}

\begin{eqnarray}\label{Euler}
\sum\limits_{n\geq 0}E_{n}(x)\frac{t^{n}}{n!}=\frac{2.e^{xt}}{e^{t}+1}~~,\left| t\right| <\pi \ .
\end{eqnarray}

\item Appell sequence theory (Appell \cite{Appell})

\begin{eqnarray*}
\frac{d}{dx}B_{n}(x)=n.B_{n-1}(x)
\end{eqnarray*}
\begin{eqnarray*}
\frac{d}{dx}E_{n}(x)=n.E_{n-1}(x)\ .
\end{eqnarray*}

\item Umbral Calculus (Lucas \cite{Lucas})

\begin{eqnarray*}
B_{n}(x)=\left( B+x\right) ^{n}  .
\end{eqnarray*}

\item Fourier Series (Hurwitz \cite{Hurwitz})

\begin{eqnarray*}
B_{n}(x)=\frac{-(n!)}{\left( 2\pi i\right) ^{n}}\sum\limits_{0\neq k\in \Z}\frac{%
e^{2\pi ikx}}{k^{n}}~~~,~0<x<1 \ .
\end{eqnarray*}
\begin{eqnarray*}
E_{n}(x)=\frac{2.(n!)}{\left( 2\pi i\right) ^{n+1}}%
\sum\limits_{k\in \Z}\frac{e^{2\pi i\left( k+\frac{1}{2}\right) x}}{\left( k+\frac{1}{2}\right) ^{n+1}}\text{ ~~~, ~}~0<x<1\ .
\end{eqnarray*}

\item Raabe multiplication theorem \cite{Raabe} 
\begin{eqnarray*}
\sum\limits_{k=0}^{m-1}B_{n}\left( \frac{x+k}{m}\right) =m^{1-n}B_{n}(x)~~~%
\text{,}~~\forall m\geq 1\text{ , }\forall n\in \mathbb{N}  \ .
\end{eqnarray*}

\begin{eqnarray*}
\sum\limits_{k=0}^{m-1}(-1)^{k}E_{n}\left( \frac{x+k}{m}%
\right) =m^{-n}E_{n}(x)\text{ ~~~~, }\forall \ m  \geq 1\text{ odd , }\forall\ n\in \mathbb{N}\ .
\end{eqnarray*}
\end{enumerate}
We refer to Lehmer's paper \cite{Lehmer}  for concise details about those approaches. All these Approachs can be generalized to any generalized Appell sequences polynomials. We omit  this point and leave the details to the reader. \\
 Our goal, in this paper, we investigate new Approach by means a Symmetry relation.  To our knowledge this approach has not yet been exploited.

\section{ Statement of main results}
For this section we consider fixed type $(g,\varphi)$ Appell sequences of polynomials.
\subsection{Characterization of Appell sequences polinomials of type $(g,\varphi)$}
We state our first main result.
\begin{theorem}[First main result]\label{main1}
Let  $a$ be  real parameter. We set \[ h(t):=f(t)e^{\frac{a}2 g(t)}\]  and   denote by 
\[ V{(a)}:=\{ (P_k)_k    \textrm{ Appell polynomials sequence   \eqref{Bernoulli-type} } \mid P_k(a-x)=(-1)^kP_k(x)\} .
\]  We have
\[ V{(a)}\neq \emptyset \iff   g  \textrm{ is odd, and }  h \textrm{ is even}.\]
\end{theorem}
\begin{proof}
We have  $V{(a)}\neq \emptyset $ if and only if  there is $(P_k)_k$  sequence of Appell polynomials  such that
\begin{eqnarray}\label{type3} 
\sum\limits_{n\geq 0}P_{n}(a-x)\frac{t^{n}}{\varphi(0)\cdots\varphi(n)}=\sum\limits_{n\geq 0}(-1)^nP_{n}(x)\frac{t^{n}}{\varphi(0)\cdots\varphi(n)}
\end{eqnarray}
The relation \eqref{type3} is equivalent to 
\[
f(t)e^{(a-x) g(t)}=f(-t)e^{xg(-t)}\ , \ \forall \ x ,
\]
thus,  $V{(a)}\neq \emptyset $ if and only if  $
h(t)=h(-t)e^{x(g(t)+g(-t))}\ , \ \forall \ x $. \\ 

Then  $V{(a)}\neq \emptyset $ if and only if   $g$ is odd  and $h$ is even.\\
\end{proof}
\begin{corollary}
For $a=0$ we have the function $h(t)=f(t).$ Then  $V{(0)}\neq \emptyset $ if and only if   $g$ is odd  and $f$ is even.
\end{corollary}
\begin{theorem} Let $a\neq 0$ be real parameter.
 We have the following characterization for the 
set $V(a)$ of Appell  sequences of polynomials.  
We have  $V{(a)}\neq \emptyset$  if and only if the functions $g$ and  $t\to (e^{ag(t)}-1) f(t)$ are odd.
\end{theorem}
\begin{proof}
Write $f_{+}$, $f_{-}$   the even and odd part of $f$, respectively.   Thanks to Theorem \ref{main1} we explore the parity of the functions $h$ and $g$. We have
\[ h(t)=(f_{+}(t)+f_{ -}(t)) e^{\frac{a}{2} g(t)}=h(-t)=(f_{+}(t)-f_{-}(t))e^{-\frac{a}{2}g(t)}\ ,
\]
which is equivalent to   $f(t) (1-e^{ag(t)})= 2f_{-}(t)$ . Then the functions $g$ and  $t\to (e^{ag(t)}-1) f(t)$ are odd.\\
Conversely, put $\psi(t)=  (e^{ag(t)}-1) f(t) $ which is odd function. Then  we have
\[\sum_{n\geq 0}P_n(a-x) (a-x)\frac{t^{n}}{\varphi(0)\cdots\varphi(n)}=f(t)e^{(a-x)g(t)}\]
and 
\begin{eqnarray}
\begin{split}
f(t)e^{(a-x)g(t)} &=&\frac{\psi(t)}{e^{ag(t)}-1}e^{(a-x)g(t)}\\
&=& \frac{\psi(-t)}{e^{ag(-t)}-1}e^{xg(-t)}\\
&=&f(-t)e^{xg(-t)}
\end{split}
\end{eqnarray}
hence \[ P_k(a-x)=(-1)^kP_k(x) , \  \forall \ k.\]
Thus we obtain our desired result. 
\end{proof}
\subsection{Application to type $g(t)=t$ }
For this section we fix the type $g(t)=t$ and $\varphi(0)=1, \varphi(n)=n, \  n\geq 1.$ 
Next we will describe the set $V(a)$  explicitly, by  truncating the  Appell  sequences \eqref{Bernoulli-type}. Denote by 
\[ V_n(a)=\{ P\in\C_n[x] | \exists (P_k)_k\in V(a): P=P_{k_0} \textrm{  for some }  k_0\in\N \}.
\]

\begin{theorem}[Second main result]\label{main1-1}Let $n$ be a positive integer, and $0\neq a$ real parameter.    We have 
\begin{eqnarray*}
V_n(a)=\textrm{ Vect}(B_{n-2k}(x/a); 0\leq k\leq n/2)
\end{eqnarray*}
is $V_n(a)$ is  the subspace  spanned by  $\{B_{n-2k}(x/a); 0\leq k\leq n/2\},$ 
with  dimension equal to $[n/2]+1, $
alternatively,
\begin{eqnarray*}
V_n(a)=\textrm{ Vect}(E_{n-2k}(x/a); 0\leq k\leq n/2).
\end{eqnarray*}
\end{theorem} 
By symmetry properties of Bernoulli and Euler polynomials it is easy to see that $V_n(a)$  contains  $\textrm{ Vect}(B_{n-2k}(x/a); 0\leq k\leq n/2)$ and $\textrm{ Vect}(E_{n-2k}(x/a); 0\leq k\leq n/2).$
To prove the converse we need more preliminaries.\\
We start to prove the following two theorems.

\begin{theorem}\label{main1-2}
Let $a$ be a nonzero real parameter, and $(P_n(x))_n$ be a sequence of  Appell polynomials  of type $(g,\varphi)$  
such that 
\begin{eqnarray}\label{Symmetry}
P_n(a-x)=(-1)^nP_n(x).
\end{eqnarray}
Let $(a_k)_{ k\in \N}$ be  sequence of real numbers such that the function
\begin{eqnarray}
F: t\to f(t)-\sum_{k}a_k\frac{t^k}{k!}  \textrm{ is odd or even}
\end{eqnarray}
Then we obtain
\begin{eqnarray}
&P_n(x)=&\sum_{k \textrm{ even } }a_k\binom{n}{k}a^{n-k}E_{n-k}\left(\frac{x}{a}\right)\ , \textrm{ if } F \textrm{ is odd }\label{3.3}\\
&P_n(x)=&-2\sum_{k \textrm{ odd } }a_k\ \frac1k \binom{n}{k-1}a^{n-k+1}B_{n-k+1}\left(\frac{x}{a}\right)\ , \textrm{ if } F \textrm{ is even }.\label{3.4}
\end{eqnarray}
\end{theorem}
\begin{proof}
If the relation  $P_n(a-x)=(-1)^nP_n(x)$ and \eqref{Bernoulli-type}  hold, we obtain 
\begin{eqnarray}\label{equ1}
f(t)e^{at}=f(-t)
\end{eqnarray}
and hence we have
\begin{eqnarray*}
  F(t)e^{at}+\displaystyle\sum_{k}a_k\frac{t^k}{k!} e^{at} -\sum_{k}a_k\frac{(-t)^k}{k!}  = F(-t) .
\end{eqnarray*}
If  $F$ is odd, then we have
\begin{eqnarray*}
    F(t) = \sum_{k}a_k\frac{t^k}{k!} \left[  \frac{(-1)^k-e^{a t}}{e^{at}+ 1}\right].
\end{eqnarray*}
Therefore we obtain
\begin{eqnarray*}
    f(t) &=& \sum_{k}a_k\frac{t^k}{k!} \left[  \frac{(-1)^k+1}{e^{at}+ 1}\right],\\
   &=& \sum_{k \textrm{  even} }a_k\frac{t^k}{k!}   \frac{2}{e^{at}+ 1}.
\end{eqnarray*}
Thus 
\begin{eqnarray*}
    f(t) e^{xt}= \sum_{k \textrm{ even }  }a^{-k}a_k\frac{(at)^k}{k!} \cdot  \frac{2e^{\frac{x}{a} (at)}}{e^{at}+ 1}.
\end{eqnarray*}
By use of the equations \eqref{Euler} and \eqref{Bernoulli-type},  we complete the proof of identity \eqref{3.3}.\\
Similarly, from the equations \eqref{Bernoulli}, \eqref{Bernoulli-type} we prove  the identity \eqref{3.4} and also get
\begin{eqnarray*}
    f(t) e^{xt}= -2\sum_{k \textrm{ odd }  }a^{-k+1}a_k\frac{(at)^{k-1}}{k!} \cdot  \frac{(at)e^{\frac{x}{a}(at)}}{e^{at}- 1}.
\end{eqnarray*}
\end{proof}
\subsection{Fourier expansions for Appell polynomials of type $g(t)=t$}
For  $0<x< 1$  if $n=1$, $ 0\leq x\leq 1$  if $n\geq 2.$  It is well-known that  \begin{eqnarray}\label{Fourier-1}
B_n(x)= \frac{-n!}{(2\pi i)^{n}}\displaystyle\sum_{k\in \mathbb {Z}\backslash\{0\}}\frac{e^{2\pi ikx}}{k^{n}}.
\end{eqnarray}
Concerning Euler's polynomials,  for  $0<x< 1$  if $n=0$, $ 0\leq x\leq 1$  if $n\geq 1.$  We have 
\begin{eqnarray} \label{Fourier-2}
E_n(x)= \frac{2(n!)}{(2\pi i)^{n+1}}
\sum_{k\in \mathbb {Z}}\frac{e^{2\pi i\left(k-\frac12\right)x}}{{\left(k-\frac12\right)}^{n+1}}.
\end{eqnarray}
\eqref{Fourier-1} and \eqref{Fourier-2}.
\begin{theorem}\label{Fourier-main1-3}
Let $a$ be a nonzero real parameter, and $(P_n(x))_n$ be a sequence of  Appell polynomials  of type $(g,\varphi)$ such that 
\begin{eqnarray}\label{Symmetry}
P_n(a-x)=(-1)^nP_n(x).
\end{eqnarray}
Let $(a_k)_{ k\in \N}$ be  sequence of real numbers such that the function
\begin{eqnarray}
F: t\to f(t)-\sum_{k}a_k\frac{t^k}{k!}  \textrm{ is odd or even}
\end{eqnarray}
\begin{enumerate}
\item For   $F$  odd ,  $0<x/a< 1$  if $n=0$, $ 0\leq x/a\leq 1$  if $n\geq 1$,  and  write 
\[c_m^{-}(a):= \sum_{k \textrm{ even } } \frac{a_k}{k!} \left( \frac{\pi i}{a} \right)^k (2m-1)^k.\]
  Then we have
\begin{eqnarray}
P_n(x)=\frac{2a^n(n!)}{(2\pi i)^{n+1}}\sum_{m\in\mathbb Z} c_m^{-}(a) \frac{e^{2\pi i\left(m-\frac12\right)x}}{{\left(m-\frac12\right)}^{n+1}}. 
\end{eqnarray}
\item For  $F$ even,  $0<x/a< 1$  if $n=1$, $ 0\leq x/a\leq 1$  if $n\geq 2,$ and  write 
\[c_m^{+}(a):= \sum_{k \textrm{ odd } } \frac{a_k}{k!}  \left( \frac{\pi i}{a} \right)^{k-1} m^{k-1}.\]
  Then we have
  \begin{eqnarray}
P_n(x)=-\frac{2a^n(n!)}{(2\pi i)^{n+1}}\sum_{m\in\mathbb Z} c_m^{+}(a) \frac{e^{2\pi i\ m x}}{{ m}^{n}}. 
\end{eqnarray}
\end{enumerate}
\end{theorem}
\section{New results  on Bernoulli and Euler polynomials of higher order}
In this section, we give two applications of our results. We obtain new explicit formulas and Fourier series for Bernoulli and Euler polynomials of higher order.

\subsection{Bernoulli polynomials}\

 We start with two applications, we obtain new characterizations of Bernoulli and Euler polynomials. Note that for  $\varphi(0)=1, \varphi(k)=k,\ k\geq 1$, $ a_k=E_k(0)$, is well-known that  $E_0(0)=1$ and  $E_k(0)=0$  for $0\neq k$ {\it even}. Then   we obtain
\begin{eqnarray*}
    f(t) e^{xt}= \sum_{k \textrm{ even }  }E_k(0)\frac{t^k}{k!} \cdot  \frac{2e^{xt}}{e^t+ 1}= \frac{2e^{xt}}{e^t+ 1}\ \cdot
\end{eqnarray*}
We get  $P_n(x)=E_n(x)$. It means, in particular,  that if the function $F: t\to f(t)-1$ is { \it odd } and $P_n(1-x)=(-1)^nP_n(x)$ then $P_n(x)=E_n(x).$  Similarly, one can apply it to Bernoulli polynomials $B_n(x)$.\\
 We have the following  general formulation.
\begin{theorem}[Bernoulli polynomials by symmetry]  Let $(P_n(x))_n$ be a sequence of Bernoulli's polynomials  type \eqref{Bernoulli-type} such that 
\begin{eqnarray}\label{Symmetry}
P_n(1-x)=(-1)^nP_n(x).
\end{eqnarray}
Let $N$ be a positive integer. If the function  $F: t\to f(t)-\displaystyle\sum_{k=0}^NB_k(0)\frac{t^k}{k!} $   is even,  then we obtain
\begin{eqnarray}
P_n(x)=B_{n}\left(x\right) ,\ \textrm{ and  }  f(t)=\frac{t}{e^t-1}.
\end{eqnarray}
\end{theorem}
\begin{proof}
We apply Theorem \ref{main1-2}  and equation \eqref{3.4},  with  $a=1$, $ a_k=B_k(0)$.   It is well-known that  $B_1(0)=-1/2$ and  $B_k(0)=0$  for $0\neq k$ {\it odd}. Then   we obtain
\begin{eqnarray*}
    f(t) e^{xt}=-2 \sum_{k \textrm{ odd }  }B_k(0)\frac{t^k}{k!} \cdot  \frac{te^{xt}}{e^t-1}= \frac{te^{xt}}{e^t-1}\ \cdot
\end{eqnarray*}
 It means, in particular,  that  the function $F: t\to f(t)+\frac{t}2$ is { \it  even } and $P_n(1-x)=(-1)^nP_n(x)$.

Thus we get  $P_n(x)=B_n(x).$ 
Then we deduce the theorem.
\end{proof}
\subsection{Euler  polynomials}\
\begin{theorem}[Euler polynomials by symmetry] Let $(P_n(x))_n$ be a sequence of Euler's polynomials  type \eqref{Bernoulli-type} such that 
\begin{eqnarray}\label{Symmetry}
P_n(1-x)=(-1)^nP_n(x).
\end{eqnarray}
Let $N$ be a positive integer. If the function  $F: t\to f(t)-\displaystyle\sum_{k=0}^NE_k(0)\frac{t^k}{k!} $   is odd,  then we obtain
\begin{eqnarray}
P_n(x)=E_{n}\left(x\right) ,\ \textrm{ and  }  f(t)=\frac{2}{e^t+1}.
\end{eqnarray}
\end{theorem}
\begin{proof}
Again we apply Theorem \ref{main1-2}  and equation \eqref{3.3},  with  $a=1$, $ a_k=E_k(0)$.   It is well-known that  $E_0(0)=1$ and   $E_k(0)=0$  for $0\neq k$ {\it even}. Then   we obtain
\begin{eqnarray*}
    f(t) e^{xt}= \sum_{k \textrm{ even }  }E_k(0)\frac{t^k}{k!} \cdot  \frac{2e^{xt}}{e^t+1}= \frac{2e^{xt}}{e^t+1}\ \cdot
\end{eqnarray*}
 It means, in particular,  that  the function $F: t\to f(t)-1 $ is { \it  odd } and $P_n(1-x)=(-1)^nP_n(x)$.

Thus we get  $P_n(x)=E_n(x).$ 
Then  we complete the proof of  the theorem.
\end{proof}

\subsection{Bernoulli polynomials of order $r$}  Let $r$ be a positive integer. The Bernoulli polynomials and numbers of order $r$ are given through 
the equations
\[ \sum_{n\geq 0}B_n^{(r)}(x)\frac{t^n}{n!}=\left(\frac{t}{e^{t}-1}\right)^{r} e^{xt}\] and 
$B_n^{(r)}=B_n^{(r)}(0).$\\
The function $f(t)=\left(\frac{t/r}{e^{t/r}-1}\right)^{r}$, $r\geq 1$  satisfy $f(t) e^t=f(-t)$ and the  function
\[F(t)=f(t)-\sum_{k\geq 0}  r^{-2k-1} B_{2k+1}^{(r)} \frac{t^{2k+1}}{(2k+1)!}\] is even function. So that we obtain
\[r^{-n}B_n^{(r)}(rx)=-2\sum_{0\leq k\leq n/2}  r^{-2k-1} \frac{B_{2k+1}^{(r)}}{2k+1} \binom{n}{2k}
B_{n-2k}(x).
\]

We have a new formula for the  generalized Bernoulli  polynomials $B_n^{(r)}(x)$ 
\[B_n^{(r)}(x)=-2\sum_{0\leq k\leq n/2}  r^{n-2k-1} \frac{B_{2k+1}^{(r)}}{2k+1} \binom{n}{2k}
B_{n-2k}(x/r)
\]
and , thanks to the relations in \cite[(1.10), Corollary 1.8]{Ba-Be},  the coefficients $B_{2k+1}^{(r)}$ are given by the formula   
\begin{eqnarray}\label{red1}
B_{n}^{(r)} =\begin{cases}
n\binom{n-1}{r-1}\sum_{k=1}^{r}  (-1)^{k-1}s(r,k)\frac{B_{n-r+k} }{n-r+k}\ ,&n\geq r \\
& \\
\frac1{\binom{r-1}{n}}s(r,r-n), &0\leq n\leq r-1\ .
\end{cases}
\end{eqnarray}
where  $s(n,l)$ is  the Stirling number of the first  kind.\\
Therefore we obtain the formulas
\begin{theorem} \label{higher-Bernoulli}  
For $n\geq r\geq 2$, we have  
\begin{eqnarray}
\begin{split}
B_n^{(r)}(x)=-2\ r^{n-1}\sum_{0\leq k\leq r/2-1}  r^{-2k} \ 
\frac{s(r,r-2k-1)}{2k+1}\frac{ \binom{n}{2k}}{\binom{r-1}{2k+1} }
B_{n-2k}(x/r)\\
-2\ r^{n-1}\sum_{r/2\leq k\leq n/2}  r^{-2k} \frac{B_{2k+1}^{(r)}}{2k+1} \binom{n}{2k}B_{n-2k}(x/r),
\end{split}
\end{eqnarray}
with \[ \frac{B_{2k+1}^{(r)} }{2k+1}=\binom{2k}{r-1}\sum_{j=1}^{r}  (-1)^{j-1}s(r,j)\frac{B_{2k+1-r+j} }{2k+1-r+j}\ ,j\geq r/2 .\]
\end{theorem}

We give details for  $r=2,3$ which give us new explicit  formulas for  $B_n^{(2)}(x)$ and $B_n^{(3)}(x).$\\ 
Thanks to $B^{(2)}_1=s(2,1)=-1, B_{2k+1}^{(2)}=-(2k+1)B_{2k} ,\ (k\geq 1) $, then we have
\begin{eqnarray*}
B_n^{(2)}(x)&=& 
\sum_{0\leq k\leq n/2}  2^{n-2k} \binom{n}{2k}B_{2k}B_{n-2k}(x/2), \\
B_n^{(3)}(x)&=&3^nB_n(x/3)-2\sum_{1\leq k\leq n/2}  3^{n-2k-1} \frac{B_{2k+1}^{(3)}}{2k+1} \binom{n}{2k}B_{n-2k}(x/3)
\end{eqnarray*}
and then  for $n\geq 4$ we have
\begin{eqnarray*}
B_n^{(3)}(x)=3^nB_n(x/3)+\frac12 3^{n-2}\binom{n}{2}B_{n-2}(x/3)-2\sum_{2\leq k\leq n/2}3^{n-2k} (2k-1)\binom{n}{2k}B_{2k}B_{n-2k}(x/3) . 
\end{eqnarray*}
\subsection{Euler polynomials of order $r$}    Let $r$ be a positive integer. The Euler polynomials and numbers of order $r$ are given by
the equations
\[ \sum_{n\geq 0}E_n^{(r)}(x)\frac{t^n}{n!}=\left(\frac{2}{e^{t}+1}\right)^{r} e^{xt}\] and 
$E_n^{r}=E_n^{(r)}(0).$\\

Let $f(t)=\left(\frac{2}{e^{t/r}+1}\right)^{r}$.   The function $f$ satisfy $f(t) e^t=f(-t)$ and the  function
\[F(t)=f(t)-\sum_{k\geq 0}  r^{-2k} E_{2k}^{(r)} \frac{t^{2k}}{(2k)!}\] is odd function. So that we obtain
\[r^{-n}E_n^{(r)}(rx)=\sum_{0\leq k\leq n/2}  r^{-2k} E_{2k}^{(r)}\binom{n}{2k}
E_{n-2k}(x).
\]

Then we get  a new formula  $E_n^{(r)}(x)$ 
\[E_n^{(r)}(x)=\sum_{0\leq k\leq n/2}  r^{n-2k} \binom{n}{2k} E_{2k}^{(r)}
E_{n-2k}(x/r)
\]

On the other hand, it is well-known that  the numbers $E_{n}^{(r)}$ can be express explicitly in terms of Stirling numbers of first kind and Euler numbers $E_k:=E_k(0).$ Precisely, we have
\begin{lemma}  Let $r$ be a positive integer. We have 
\begin{eqnarray}
E_n^{(r)}=\frac{2^{r-1}}{(r-1)!}\sum_{j=0}^{r-1}(-1)^j s(r,r-j)E_{n+r-j-1}.
\end{eqnarray}
\end{lemma}

Hence we get the result 
\begin{theorem}\label{higher-Euler} Let $r$ be a positive integer. We have 
\begin{eqnarray}
E_n^{(r)}(x)=\frac{2^{r-1}}{(r-1)!}\sum_{0\leq j\leq r-1\atop  0\leq k\leq n/2}(-1)^j s(r,r-j)\binom{n}{2k}r^{n-2k}  E_{2k+r-j-1} 
E_{n-2k}(x/r).
\end{eqnarray}
\end{theorem}
The relation is obvious for $r=1$. For $r=2$ and $n\geq 2$  we obtain
\begin{eqnarray}\label{Euler-order-2}
E_n^{(r)}(x)=2^nE_n(x/2)+\sum_{1\leq k\leq n/2}\binom{n}{2k}2^{n+1-2k}  E_{2k+1} 
E_{n-2k}(x/r).
\end{eqnarray}
\subsection{Fourier expansions for higher Bernoulli-Euler's polynomials}
We apply our main result to get Fourier series for Euler's and Bernoulli's polynomials of higher order $r$.  
From our Theorem \ref{higher-Bernoulli} and Theorem \ref{higher-Euler} we can obtain the Fourier expansions for  the polynomials $B_n^{(r)}(x)$ and $E_n^{(r)}(x).$
\begin{theorem}[Fourier expansion]\label{Fourier} For $x\in(0,r)$ we have Fourier expansion for the Euler polynomials of order $r\geq 1$ given by 
\begin{eqnarray}
E_n^{(r)}(x)=
 \frac{2^r}{(r-1)!}\frac{n!}{(2\pi i)^{n+1}}\displaystyle\sum_{m\in\Z}c_m(n,r) \frac{e^{2\pi i(m-1/2)\frac{x}{r}}}{(m-1/2)^{n+1}},
\end{eqnarray}
where 
\begin{eqnarray}\label{Fourier-Euler-order}
c_m(n,r)=\displaystyle\sum_{0\leq j\leq r-1\atop  0\leq k\leq n/2}(-1)^j s(r,r-j) (\pi i)^{2k}(2m-1)^{2k}  E_{2k+r-j-1}.
\end{eqnarray}
\end{theorem}
For example for
\begin{enumerate}
\item  $r=1$ we have $c_m(n,1)=1$ for any $m\in\Z$. We recover the known result about periodized Euler functions.
\item  $r=2$, 
 \begin{eqnarray}\label{Fourier-Euler-order-2} 
 c_m(n,2)=1/2+\sum_{1\leq k\leq n/2} (\pi i)^{2k} (2m-1)^{2k}  E_{2k+1}\ , \ n\geq 2. 
 \end{eqnarray}
\end{enumerate}

\end{document}